\documentclass[12pt]{amsart}
\usepackage[a4paper,margin=2.5cm]{geometry}
\usepackage{amsmath}
\usepackage{amssymb}
\usepackage{amsthm}
\usepackage{enumerate}
\usepackage{color}
\usepackage{tabu}
\usepackage{hyperref}
\newcommand{\N}{\mathbb{N}}

\newcommand{\SG}[1]{{}^{2\!}r\!\left(#1\right)}
\newcommand{\GG}[1]{{}^{k\!}r\!\left(#1\right)}

\newcommand{\RG}[1]{{}^{k\text{\tiny -power}\!}r\!\left(#1\right)}
\newcommand{\RP}[1]{{}^{P\!}r\!\left(#1\right)}
\newcommand{\SR}[1]{{}_{2}r\!\left(#1\right)}
\newcommand{\KR}[1]{{}_{k}r\!\left(#1\right)}
\newcommand{\KRR}[1]{{}_{k\text{\tiny -power}}r\!\left(#1\right)}

\newcommand{\KP}[1]{{}_{P}r\!\left(#1\right)}
\newcommand{\Frob}[1]{g\!\left(#1\right)}
\newcommand{\ub}[1]{h\!\left(#1\right)}

\newcommand{\ger}[1]{\left\langle #1 \right\rangle}
\allowdisplaybreaks
\renewcommand{\ge}{\geqslant}

\renewcommand{\le}{\leqslant}

\theoremstyle{plain}
\newtheorem{thm}{Theorem}[section]
\newtheorem{lem}[thm]{Lemma}
\newtheorem{prop}[thm]{Proposition}
\newtheorem{cor}[thm]{Corollary}
\newtheorem{conj}[thm]{Conjecture}
\newtheorem{quest}{Question}
\newtheorem{prob}{Problem}
\theoremstyle{definition}
\newtheorem{defn}[thm]{Definition}

\theoremstyle{remark}
\newtheorem{rem}{Remark}

\newtheorem{case}{Case}
\definecolor{Green}{rgb}{0.0, 0.42, 0.24}
\usepackage{soul,xcolor}

\begin{document}

\title[The square Frobenius number]{The square Frobenius number}

\author[J. Chappelon]{Jonathan Chappelon}
\address{IMAG, Univ.\ Montpellier, CNRS, Montpellier, France}
\email{\href{mailto:jonathan.chappelon@umontpellier.fr}{jonathan.chappelon@umontpellier.fr}}

\author[J.L. Ram\'irez Alfons\'in]{Jorge Luis Ram\'irez Alfons\'in}
\address{UMI2924 - Jean-Christophe Yoccoz, CNRS-IMPA, Brazil and IMAG, Univ.\ Montpellier, CNRS, Montpellier, France}
\email{\href{mailto:jorge.ramirez-alfonsin@umontpellier.fr}{jorge.ramirez-alfonsin@umontpellier.fr}}

\thanks{The second author was partially supported by INSMI-CNRS}

\date{May 9, 2022}


\begin{abstract}
Let $S=\left\langle s_1,\ldots,s_n\right\rangle$ be a numerical semigroup generated by the relatively prime positive integers $s_1,\ldots,s_n$. Let $k\geqslant 2$ be an integer. In this paper, we consider the following $k$-power variant of the Frobenius number of $S$ defined as 
$$
{}^{k\!}r\!\left(S\right):= \text{ the largest } k \text{-power integer not belonging to } S.
$$
In this paper, we investigate the case $k=2$. We give an upper bound for ${}^{2\!}r\!\left(S_A\right)$  for an infinite family of semigroups $S_A$ generated by {\em arithmetic progressions}. The latter turns out to be the exact value of ${}^{2\!}r\!\left(\left\langle s_1,s_2\right\rangle\right)$ under certain conditions. We present an exact formula for ${}^{2\!}r\!\left(\left\langle s_1,s_1+d \right\rangle\right)$ when $d=3,4$ and $5$, study ${}^{2\!}r\!\left(\left\langle s_1,s_1+1 \right\rangle\right)$ and ${}^{2\!}r\!\left(\left\langle s_1,s_1+2 \right\rangle\right)$ and put forward  two relevant conjectures. We finally discuss some related questions. 
\end{abstract}
\subjclass[2010]{Primary 11D07}
\keywords{Numerical semigroups, Frobenius number, Perfect square integer}
\maketitle

\section{Introduction}

Let $s_1, \dots, s_n$ be relatively prime positive integers. Let 
$$
S=\ger{s_1, \dots,s_n}= \left\{\sum_{i=1}^n x_is_i\ \middle|\  x_i \text{ integer, } x_i\ge 0\right\}
$$ 
be the numerical semigroup generated by $s_1, \dots, s_n$. The largest integer which is not an element of $S$, denoted by $\Frob{S}$ or $\Frob{\ger{s_1,\dots,s_n}}$, is called the {\em Frobenius number} of $S$. It is well known that $\Frob{\ger{s_1,s_2}} = s_1s_2-s_1-s_2$. However, calculating $\Frob{S}$ is a difficult problem in general.  In \cite{ramirez1} was shown that computing $\Frob{S}$ is {\em NP-hard}. We refer the reader to \cite{ramirez2} for an extensive literature on the Frobenius number.
\smallskip

Throughout this paper, the set of non-negative integers is denoted by $\N$. The non-negative integers not in $S$ are called the {\em gaps} of $S$. The number of gaps of $S$, denoted by $N(S)$ (that is, $N(S)=\# (\mathbb{N}\setminus S)$) is called the {\em genus} of $S$. We recall that the {\em multiplicity} of $S$ is the smallest positive element belonging to $S$.
\smallskip

Given a particular (arithmetical, number theoretical, etc.) {\em Property} $P$, one might consider the following two {\em $P$-type functions} of a semigroup $S$: 
$$
\hbox{$\RP{S}$:= the largest integer having property $P$ not belonging to $S$}
$$
and
$$
\hbox{$\KP{S}:=$ the smallest integer having property $P$ belonging to $S$.}
$$
Notice that the multiplicity and the Frobenius number are $P$-type functions where $P$ is the property of being a positive integer\footnote{{\em $P$-type functions} were introduced by the second author (often mentioned during his lectures) with the hope to better understand certain properties $P$ in terms of linear forms.}.
\smallskip


In this spirit, we consider the property of being {\em perfect $k$-power} integer (that is, integers of the form $m^k$ for some integers $m,k> 1$). Let $k\ge 2$ be an integer, we define 
$$
\RG{S}:= \text{ the largest perfect } k \text{-power integer not belonging to } S.
$$

This $k$-power variant of $\Frob{S}$ is called the $k$-{\em power Frobenius number} of $S$, we may write $\GG{S}$ for short.
\smallskip

In this paper we investigate the $2$-power Frobenius number, we call it the {\em square Frobenius number}. 
\smallskip

In Section~\ref{sec:arithprogr}, we study the square Frobenius number of semigroups $S_A$ generated by {\em arithmetic progressions}. We give an upper bound for $\SG{S_A}$ for an infinite family (Theorem~\ref{theo:ap}) which turns out to be the exact value when the arithmetic progression consists of two generators (Corollary~\ref{cor:ap}).  
\smallskip

In Section~\ref{sec:2gen345}, we present exact formulas for $\SG{\ger{a, a+3}}$ where $a\ge 3$ is an integer not divisible by $3$ (Theorem~\ref{thm:2gen3}), for $\SG{\ger{a, a+4 }}$ where $a\ge 3$ is an odd integer (Theorem~\ref{thm:2gen4}) and for $\SG{\ger{a, a+5}}$ where $a\ge 2$ is an integer not divisible by $5$ (Theorem~\ref{thm:2gen5}).
\smallskip

In Sections~\ref{sec:2gen1} and \ref{sec:2gen2}, we turn our attention to the cases $\ger{a, a+1}$ where $a\ge2$ and $\ger{a,a+2}$ where $a\ge 3$ is an odd integer. We present formulas for the corresponding square Frobenius number in the case when neither of the generators are square integers (Propositions~\ref{prop:2gen1} and \ref{prop:2gen2}). We also put forward two conjectures on the values of $\SG{\ger{a,a+1}}$ and $\SG{\ger{a,a+2}}$ in the case when one of the generators is a square integer (Conjectures~\ref{conj:2gen1} and \ref{conj:2gen2}). The conjectured values have an unexpected close connection with a known recursive sequence (Equation~\eqref{eq:recursive}) and in which $\sqrt 2$ and $\sqrt 3$ (strangely) appear. A number of computer experiments support our conjectures.
\smallskip

Finally, Section~\ref{sec:concluding} contains some concluding remarks.

\section{Arithmetic progression}\label{sec:arithprogr}

Let $a$, $d$ and $k$ be positive integers such that $a$ and $d$ are relatively prime. Throughout this section, we denote by $S_A$  the semigroup generated  by the {\em arithmetic progression} whose first element is $a$, with common difference $d$ and of length $k+1$, that is,
$$
S_A =  \ger{a,a+d,a+2d,\ldots,a+kd}.
$$
Note that the integers $a,a+d,\ldots,a+kd$ are relatively prime if and only if $\gcd(a,d)=1$.

We shall start by giving a necessary and sufficient condition for a square to belong to $S_A$.

For any integer $x$ coprime to $d$, a {\em multiplicative inverse modulo $d$} of $x$ is an integer $y$ such that $xy\equiv1\bmod{d}$.

\begin{prop}\label{prop2}
Let $i$ be an integer and let $\lambda_i$ be the unique integer in $\left\{0,1,\ldots,d-1\right\}$ such that $\lambda_ia+i^2\equiv0\bmod{d}$. In other words, the integer $\lambda_i$ is the remainder in the Euclidean division of $-a^{-1}i^2$ by $d$, where $a^{-1}$ is a multiplicative inverse of $a$ modulo $d$. Then, 

$$(a-i)^2 \in S_A \text{ if and only if }
(i+kd)^2\le \left(\left(\left\lfloor\frac{i^2+\lambda_ia}{ad}\right\rfloor+k\right)d-\lambda_i\right)(a+kd).
$$
\end{prop}

A key step for the proof of this result is the following lemma, which can be thought as a variant of a result given in \cite{Rob}, see \cite[Lemma 1]{Rod} for a short proof. The arguments for the proof of this variant are similar to those used in the latter.

\begin{lem}\label{prop1}
Let $M$ be a non-negative integer and let $x$ and $y$ be the unique integers such that $M=ax+dy$, with $0\le y\le a-1$. Then, 

$$M\in S_A \text{ if and only if } y\le kx \text{ (with } x\ge 0 \text{)}.$$
\end{lem}

\begin{proof}
First, suppose that $M\in S_A$ and let $x_0,x_1,\ldots,x_k$ be non-negative integers such that $M=\sum_{i=0}^{k}x_i(a+id)$. Then, we have that
$$
M = \sum_{i=0}^{k}x_i a + \sum_{i=0}^{k}ix_i d = x'a+y'd,
$$
with $x'=\sum_{i=0}^{k}x_i\in\N$ and $y'=\sum_{i=0}^{k}ix_i\in\N$. It follows that
$$
y' = \sum_{i=0}^{k}ix_i \le k\sum_{i=0}^{k}x_i = kx'.
$$
Moreover, since $M=xa+yd$ with $y\in\{0,1,\ldots,a-1\}$, we obtain that there exists a non-negative integer $\lambda$ such that
$$
y'=y+\lambda a \quad\text{and}\quad x'=x-\lambda a.
$$
This leads to the inequality
$$
y=y'-\lambda a\le kx'-\lambda a = kx-\lambda(k+1)a \le kx.
$$
Conversely, suppose now that $y\le kx$. Obviously, since $y\ge0$, we know that $x\ge0$. Let
$$
y=qk+r
$$
be the Euclidean division of $y$ by $k$, with $q\in\N$ and $r\in\{0,1,\ldots,k-1\}$. If $r=0$, then we have that $0\le q\le x$ since $y=qk\le kx$. It follows that
$$
M = xa+qkd = (x-q)a + q(a+kd) \in S_A.
$$
Finally, if $r>0$, then we have that $0\le q\le x-1$ since $y=qk+r\le kx$. It follows that
$$
M = xa+(qk+r)d = (x-q-1)a + q(a+kd) + (a+rd) \in S_A.
$$
This completes the proof.
\end{proof}

We may now prove Proposition~\ref{prop2}.

\begin{proof}[Proof of Proposition~\ref{prop2}]
Let $i$ be an integer and let $\lambda_i\in\{0,1,\ldots,d-1\}$ such that $\lambda_ia+i^2\equiv0\bmod{d}$. We have that
$$
(a-i)^2 \begin{array}[t]{l}
= \displaystyle (a-2i)a + i^2\\ 
= \displaystyle (a-2i-\lambda_i)a + \frac{i^2+\lambda_ia}{d}d \\ \ \\
= \displaystyle\left(a-2i-\lambda_i+\left\lfloor\frac{i^2+\lambda_ia}{ad}\right\rfloor d\right)a + \left( \frac{i^2+\lambda_ia}{d}-\left\lfloor\frac{i^2+\lambda_ia}{ad}\right\rfloor a \right)d. \\
\end{array}
$$
We thus have, by Lemma~\ref{prop1}, that the square $(a-i)^2$ is in $S_A$ if and only if
$$
\begin{array}{ll}
 & \displaystyle\frac{i^2+\lambda_ia}{d}-\left\lfloor\frac{i^2+\lambda_ia}{ad}\right\rfloor a \le k\left( a-2i-\lambda_i+\left\lfloor\frac{i^2+\lambda_ia}{ad}\right\rfloor d \right) \\ \ \\
\Longleftrightarrow & \displaystyle\frac{i^2+\lambda_ia}{d} \le k\left( a-2i-\lambda_i\right) + \left\lfloor\frac{i^2+\lambda_ia}{ad}\right\rfloor (a+kd) \\ \ \\ 
\Longleftrightarrow & \displaystyle i^2+\lambda_ia \le kd\left( a-2i-\lambda_i\right) + \left\lfloor\frac{i^2+\lambda_ia}{ad}\right\rfloor d(a+kd) \\ \ \\ 
\Longleftrightarrow & \displaystyle i^2+ 2ikd \le kda -\lambda_i(a+kd) + \left\lfloor\frac{i^2+\lambda_ia}{ad}\right\rfloor d(a+kd) \\ \ \\ 
\Longleftrightarrow & \displaystyle i^2+ 2ikd +k^2d^2\le kd(a+kd) -\lambda_i(a+kd) + \left\lfloor\frac{i^2+\lambda_ia}{ad}\right\rfloor d(a+kd) \\ \ \\ 
\Longleftrightarrow & \displaystyle (i+kd)^2\le \left(\left(\left\lfloor\frac{i^2+\lambda_ia}{ad}\right\rfloor+k\right)d-\lambda_i\right)(a+kd). 
\end{array}
$$
This completes the proof.
\end{proof}

\begin{rem}\label{rem1}
We have that $\lambda_0 = 0$ and $\lambda_i>0$ for all integers $i$ such that $\gcd(i, d) = 1$ with $d\ge 2$. Moreover, $\lambda_i = \lambda_{d-i}$ for all $i\in\{1,2,\dots, d-1\}$.
\end{rem}

The above characterization permits us to obtain an upper-bound of $\SG{S_A}$ when $a$ is larger enough compared to $d\ge3$.

\begin{defn}
Let $\lambda^*$ be the integer defined by
$$
\lambda^*=\max\limits_{0\le i\le d-1}\left\{\lambda_i\in \{0,1,\ldots,d-1\}\ \mid  \lambda_i a+i^2\equiv 0\bmod{d} \right\}.
$$
Let $\left\{\alpha_1<\ldots <\alpha_n\right\}\subseteq \left\{0,1,\ldots,d-1\right\}$ such that $\lambda_{\alpha_j}=\lambda^*$ and take $\alpha_{n+1}=d+\alpha_1$. Let $j\in\left\{1,\ldots,n\right\}$ be the index such that
\begin{equation}\label{eq:definition}
{\left(\mu d+\alpha_j\right)}^2\le (kd-\lambda^*)(a+kd)<{\left(\mu d+\alpha_{j+1}\right)}^2,
\end{equation}
for some integer $\mu\ge 0$. 
\end{defn}

\begin{rem}\label{rem2}
\begin{enumerate}[(a)]
\item[]
\item
The above index $j$ exists and it is unique. Indeed, we clearly have that there is an integer $\mu$ such that
$$\mu d\le \sqrt{(kd-\lambda^*)(a+kd)}<(\mu+1) d.$$
Since $0\le \alpha_1<\cdots <\alpha_n\le d-1$, then the interval $[\mu d,(\mu+1)d[$ can be refined into intervals of the form $[\mu d+\alpha_i,\mu d+\alpha_{i+1}[$ for each $i=1,\dots ,n-1$. Therefore, there is a unique index $j$ verifying equation \eqref{eq:definition}.
\item
We have that $\mu d +\alpha_{n+1}=(\mu+1)d+\alpha_1$.
\end{enumerate}
\end{rem}

The following two propositions give us useful information on the sequence of indices $\alpha_1,\dots ,\alpha_n$.

\begin{prop}\label{prop3}
We have that $\alpha_i+\alpha_{n+1-i}=d$, for all $i\in\{1,\ldots,n\}$.
\end{prop}

\begin{proof}
Since $\left\{i\in\{1,\ldots,n\}\ \middle|\ \lambda_i= \lambda^*\right\}=\left\{\alpha_1,\ldots,\alpha_n\right\}$, with $\alpha_1<\alpha_2<\cdots<\alpha_n$, and since $\lambda_{d-i}=\lambda_i$, for all $i\in\left\{1,\ldots,d-1\right\}$, by Remark~\ref{rem1}.
\end{proof}

\begin{prop}\label{prop4}
If $d\ge3$ then $n\ge2$ and $1\le\alpha_1<\frac{d}{2}<\alpha_n\le d-1$.
\end{prop}

\begin{proof}
Suppose that $n=1$ and hence $\alpha_n=\alpha_1$. Since $d=\alpha_1+\alpha_n=2\alpha_1$ by Proposition~\ref{prop3}, it follows that $d$ is even and $\alpha_1=\frac{d}{2}$.

If $d$ is divisible by $4$ then
$$
\left(\frac{d}{2}\right)^2 = \frac{d}{4}\cdot d \equiv 0 \pmod{d}.
$$
Therefore, $\lambda^*=\lambda_{\alpha_1}=\lambda_{\frac{d}{2}}=0$. Moreover, since $\gcd(1,d)=1$, we know that $\lambda_1>0$. It follows that $\lambda_1>\lambda^*$, in contradiction with the maximality of $\lambda^*$.

If $d$ is even, not divisible by $4$, then $\frac{d}{2}$ is odd and
$$
\left(\frac{d}{2}\right)^2 = \frac{d}{2}\cdot\frac{d}{2} = \frac{\frac{d}{2}-1}{2}d+\frac{d}{2} \equiv \frac{d}{2}\pmod{d}.
$$
Since $a$ is coprime to $d$, we know that there exists a multiplicative inverse $a^{-1}$ modulo $d$ such that $aa^{-1}\equiv 1\bmod{d}$. Since $d$ is even, it follows that $a^{-1}$ is odd and we obtain that
$$
\lambda_{\frac{d}{2}} \equiv -a^{-1}\left(\frac{d}{2}\right)^2 \equiv -a^{-1}\frac{d}{2} \equiv \frac{d}{2} \pmod{d}.
$$
Therefore, $\lambda_{\frac{d}{2}}=\frac{d}{2}$. Moreover, for any $i\in\left\{0,\ldots,\frac{d}{2}-1\right\}$, since
$$
\left(i+\frac{d}{2}\right)^2 = i^2+id+\left(\frac{d}{2}\right)^2 \equiv i^2+\frac{d}{2} \pmod{d}
$$
and since $a^{-1}$ is odd, it follows that
$$
\lambda_{i+\frac{d}{2}} \equiv -a^{-1}\left(i+\frac{d}{2}\right)^2 \equiv -a^{-1}i^2-a^{-1}\frac{d}{2} \equiv \lambda_i + \frac{d}{2} \pmod{d},
$$ 
for all $i\in\left\{0,\ldots,\frac{d}{2}-1\right\}$. Since $d\ge3$, we have that $1<\frac{d}{2}<1+\frac{d}{2}<d$. Finally, since $\lambda_1>0$, we deduce that
$$
\max\left\{\lambda_1,\lambda_{1+\frac{d}{2}}\right\} > \frac{d}{2} = \lambda_{\frac{d}{2}},
$$
in contradiction with the maximality of $\lambda_{\frac{d}{2}}$.

We thus have that if $d\ge3$ then $n\ge2$ and $\alpha_1<\alpha_n$. Since $\alpha_1+\alpha_n=d$, by Proposition~\ref{prop3}, we deduced that $\alpha_1<\frac{d}{2}<\alpha_n$. This completes the proof.
\end{proof}

\begin{defn}
Let us now consider the integer function $\ub{a,d,k}$ defined as
$$
\ub{a,d,k} := {\left(a-\left((\mu-k)d+\alpha_{j+1}\right)\right)}^2.
$$
\end{defn}

\begin{rem} We notice that the function $\ub{a,d,k}$ can always be computed for any relatively prime integers $a$ and $d$ and any positive integer $k$. It is enough to calculate $\lambda_i$ for each $i=0,\dots ,d-1$, from which $\lambda^*$ and the set of $\alpha_i$'s can be obtained and thus the desired $\mu$ and $\alpha_{j+1}$ can be computed. 
\end{rem}


\begin{thm}\label{theo:ap}
Let $d\ge3$ and $a+kd\ge 4kd^3$. Then,
$$
\SG{S_A} \le \ub{a,d,k}.
$$
\end{thm}

We need the following lemma before proving Theorem~\ref{theo:ap}.

\begin{lem}\label{lem1}
If $d\ge3$ then
$$
\alpha_{i+1}-\alpha_i\le d-1\quad\text{and}\quad \alpha_i+\alpha_{i+1}\le 2d
$$
for all $i\in\{1,\ldots,n\}$.
\end{lem}

\begin{proof}
First, let $i\in\{1,\ldots,n-1\}$. Since $1\le \alpha_j\le d-1$, for all $j\in\{1,\ldots,n\}$, from Remark~\ref{rem1}, it follows that
$$
\alpha_{i+1}-\alpha_i < \alpha_{i+1} \le d-1\quad\text{and}\quad \alpha_i+\alpha_{i+1}<2d.
$$
Finally, for $i=n$, since $n\ge2$ and $\alpha_n>\alpha_1$ by Proposition~\ref{prop4}, it follows that
$$
\alpha_{n+1}-\alpha_n = d+\alpha_1-\alpha_n <d.
$$
Moreover, since $\alpha_n=d-\alpha_1$ by Proposition~\ref{prop3}, we obtain that
$$
\alpha_n+\alpha_{n+1} = (d-\alpha_1)+(d+\alpha_1)=2d.
$$ 
This completes the proof.
\end{proof}

We now have all the ingredients to prove Theorem~\ref{theo:ap}.

\begin{proof}[Proof of Theorem~\ref{theo:ap}]
It is known \cite{Rob} that
$$
\Frob{S_A}=\left(\left\lfloor \frac{a-2}{k}\right\rfloor +1\right)a+(d-1)(a-1)-1.
$$ 
Since $a^2>\left(\left\lfloor \frac{a-2}{k}\right\rfloor +1\right)a,\ 2akd>(d-1)(a-1)$ and $(kd)^2>0$ then
$$
\Frob{S_A}<a^2+2kda+(kd)^2=(a-(-kd))^2.
$$
Therefore, it is enough to show that $(a-i)^2\in S$ for all $-kd\le i<(\mu-k)d+\alpha_{j+1}$. 

We have two cases.
\setcounter{case}{0}

\begin{case}
$-kd \le i\le (\mu-k)d+\alpha_j$.

We have that
$$
\begin{array}{lll}
(i+kd)^2 & \le (\mu d+\alpha_j)^2 & (\text{since }  i\le (\mu-k)d+\alpha_j) \\
& \le (kd-\lambda^*)(a+kd)& \text{(by definition)} \\
& \le \left(\left(\left\lfloor\frac{i^2+\lambda_ia}{ad}\right\rfloor+k\right)d-\lambda_i\right)(a+kd) & (\text{since }\lambda^*\ge \lambda_i \text { and } \left\lfloor\frac{i^2+\lambda_ia}{ad}\right\rfloor \ge0)\\
\end{array}
$$
Therefore, by Proposition~\ref{prop2}, we obtain that $(a-i)^2\in S_A$.
\end{case}

\begin{case}
$(\mu-k)d+\alpha_j<i<(\mu-k)d+\alpha_{j+1}$.

In this case we have that $\alpha_j<i\bmod{d}<\alpha_{j+1}$ implying that $\lambda_i\le \lambda^*-1$ and thus
\begin{equation}\label{eq:ap0}
(kd-\lambda_i)(a+kd) \ge (kd-\lambda^*)(a+kd)+(a+kd).
\end{equation}
Moreover,
\begin{eqnarray}\label{eq:arr1}
(i+kd)^2 & < & \left(\mu d+\alpha_{j+1}\right)^2 \\
 & = & \left(\left(\mu d+\alpha_j\right)+\left(\alpha_{j+1}-\alpha_{j}\right)\right)^2 \nonumber \\
 &= & \left(\mu d+\alpha_j\right)^2 + \left(\alpha_{j+1}-\alpha_j\right)\left(2\left(\mu d+\alpha_j\right)+\left(\alpha_{j+1}-\alpha_j\right)\right) \nonumber\\
 &= & \left(\mu d+\alpha_j\right)^2 + \left(\alpha_{j+1}-\alpha_j\right)\left(2\mu d+\alpha_j+\alpha_{j+1}\right).\nonumber
\end{eqnarray}

Now, from Lemma~\ref{lem1}, we have that
\begin{equation}\label{eq:ap1b}
\alpha_{\ell+1}-\alpha_\ell<d\quad\text{and}\quad \alpha_\ell+\alpha_{\ell+1}\le 2d
\end{equation}
for all $\ell\in\{1,\ldots,n\}$.
Therefore, combining  \eqref{eq:arr1} and \eqref{eq:ap1b}, we obtain
\begin{equation}\label{eq:ap3}
(i+kd)^2 < \left(\mu d+\alpha_j\right)^2 + d\left(2\mu d +2d\right) =  \left(\mu d+\alpha_j\right)^2 + 2d^2\left(\mu+1\right)
\end{equation}
for a $j\in\{1,\ldots,n\}$.

Since
$$
 \left(\mu d+\alpha_j\right)^2 \stackrel{(by\ definition)}{\le} \left(kd-\lambda^*\right)(a+kd) \stackrel{\eqref{eq:ap0}}{\le} \left(kd-\lambda_i\right)(a+kd) - (a+kd)
$$
then
\begin{equation}\label{eq:ap4}
(i+kd)^2 <  \left(kd-\lambda_i\right)(a+kd) +2d^2\left(\mu+1\right) - (a+kd).
\end{equation}

We claim that 
\begin{equation}\label{eq:claim}
2d^2(\mu+1)\le a+kd.
\end{equation}
 
We have two subcases
\smallskip

{\it Subcase} $i)$ For $j\in\{1,\ldots,n-1\}$. Since ${\left(\mu d+\alpha_j\right)}^2\le (kd-\lambda^*)(a+kd)<{\left(\mu d+\alpha_{j+1}\right)}^2$, $\alpha_j\ge1$ and $\alpha_{j+1}\le \alpha_n<d$ then
$$
\mu = \left\lfloor\frac{\sqrt{(kd-\lambda^*)(a+kd)}}{d}\right\rfloor.
$$
Moreover, since $a+kd\ge 4kd^3>4(kd-\lambda^*)d^2$, it follows that
\begin{equation}\label{eq;mu}
\mu\ge2(kd-\lambda^*) \ \text{with } \lambda^*>0.
\end{equation}
If $\mu=2(kd-\lambda^*)$, then we have
$$
2d^2(\mu+1) = 4kd^3+2(1-2\lambda^*)d^2 \le 4kd^3 \le a+kd, 
$$
as announced. Otherwise, if $\mu>2(kd-\lambda^*)$, it follows that
{\footnotesize $$
(kd-\lambda^*)(a+kd)\ge\left(\mu d+\alpha_j\right)^2\stackrel{(\alpha_{j}\ge 1)}{>}\mu^2d^2 > \left(\mu^2-1\right)d^2 = \left(\mu-1\right)\left(\mu+1\right)d^2 \ge 2(kd-\lambda^*)(\mu+1)d^2,
$$}
obtaining the claimed inequality \eqref{eq:claim} for $j\in\{1,2,\ldots,n-1\}$. 

\smallskip
{\it Subcase} $ii)$
For $j=n$.  Since ${\left(\mu d+\alpha_n\right)}^2\le (kd-\lambda^*)(a+kd)<{\left(\mu d+\alpha_{n+1}\right)}^2$, where $\alpha_n=d-\alpha_1$ and $\alpha_{n+1}=d+\alpha_1$, we obtain
$$
{\left((\mu+1) d-\alpha_1\right)}^2\le (kd-\lambda^*)(a+kd)<{\left((\mu+1) d+\alpha_{1}\right)}^2.
$$
Since $\alpha_1<d$, we have
$$
\left\lfloor\frac{\sqrt{(kd-\lambda^*)(a+kd)}}{d}\right\rfloor\in\left\{\mu,\mu+1\right\}.
$$
Moreover, since $a+kd\ge 4kd^3>4(kd-\lambda^*)d^2$, it follows that
\begin{equation}\label{eq;mu2}
\mu+1 \ge 2(kd-\lambda^*).
\end{equation}
If $\mu+1=2(kd-\lambda^*)$, then we have
$$
2d^2(\mu+1)=4(kd-\lambda^*)d^2 < 4kd^3\le a+kd,
$$
obtaining the claimed inequality \eqref{eq:claim}. Otherwise, if $\mu+1>2(kd-\lambda^*)$, since $\alpha_1<\frac{d}{2}$ from Proposition~\ref{prop4}, we obtain 
$$
(kd-\lambda^*)(a+kd) \begin{array}[t]{l}
\ge\displaystyle\left((\mu+1)d-\alpha_1\right)^2> \left((\mu+1) d-\frac{d}{2}\right)^2 = \left(\mu^2+\mu +\frac{1}{4}\right)d^2 \\ \ \\
> \displaystyle\mu\left(\mu+1\right)d^2 \stackrel{\mu\ge2(kd-\lambda^*)}{\ge} 2(kd-\lambda^*)(\mu+1)d^2, \\
\end{array}
$$
obtaining the claimed inequality \eqref{eq:claim} when $j=n$. 
\smallskip

Finally, since inequality \eqref{eq:claim} is true for any $j\in\{1,\ldots,n\}$ then, from equation \eqref{eq:ap4} we have
$$
(i+kd)^2 <  \left(kd-\lambda_i\right)(a+kd) +2d^2\left(\mu+1\right) - (a+kd) \le \left(kd-\lambda_i\right)(a+kd).
$$
We deduce, by Proposition~\ref{prop2}, that $(a-i)^2\in S_A$.
\end{case}
This completes the proof.
\end{proof}

\begin{rem}
The above proof can be adapted if we consider the weaker condition $a+kd>4(kd-\lambda^*)d^2+d^2$ instead of $a+kd\ge 4kd^3$.
\end{rem}

We believe that the upper bound $\ub{a,d,k}$ of $\SG{S_A}$ given in Theorem \ref{theo:ap} is actually an equality. We are able to establish the latter  in the case when $k=1$ for any $d\ge 3$.

\begin{cor}\label{cor:ap}
Let $d\ge 3$ and $a+d\ge4d^3$. Then,
$$
\SG{ \left\langle a,a+d\right\rangle} = \ub{a,d,1}.
$$
\end{cor}

\begin{proof}
By Theorem~\ref{theo:ap}, we have $\SG{ \left\langle a,a+d\right\rangle} \le {\left(a-\left((\mu-1)d+\alpha_{j+1}\right)\right)}^2$. It is thus enough to show that  ${\left(a-\left((\mu-1)d+\alpha_{j+1}\right)\right)}^2\not\in \left\langle a,a+d\right\rangle$. 

Let $i= (\mu-1)d+\alpha_{j+1}$. We have 
$$
i^2
\begin{array}[t]{l}
 = \left((\mu-1)d+\alpha_{j+1}\right)^2 \\[2ex]
 = \left(\left(\mu d+\alpha_j\right)-\left(d+\alpha_j-\alpha_{j+1}\right)\right)^2 \\[2ex]
 = \left(\mu d+\alpha_j\right)^2 - \left(d+\alpha_j-\alpha_{j+1}\right)\left(2\left(\mu d+\alpha_j\right)-\left(d+\alpha_j-\alpha_{j+1}\right)\right) \\[2ex]
 = \left(\mu d+\alpha_j\right)^2 - \left(d+\alpha_j-\alpha_{j+1}\right)\left((2\mu-1)d+\alpha_j+\alpha_{j+1}\right).
\end{array}
$$
Since $d+\alpha_j-\alpha_{j+1}\ge 1$, by Lemma~\ref{lem1}, and $(2\mu-1)d+\alpha_j+\alpha_{j+1}>(2\mu-1)d$, it follows that
\begin{equation}\label{eq:1aa}
i^2 < \left(\mu d+\alpha_j\right)^2 - (2\mu-1)d \le (d-\lambda^*)(a+d) - (2\mu-1)d.
\end{equation}

Since $a+d\ge 4d^3$, we already know that $\mu+1\ge 2(d-\lambda^*)$ (see equations \eqref{eq;mu} and \eqref{eq;mu2} with $k=1$). It follows that $\mu\ge2(d-\lambda^*)-1\ge1$ and then 
\begin{equation}\label{eq:1ac}
d-\lambda^* \le \frac{\mu+1}{2} \le 2\mu-1.
\end{equation}

By combining equations \eqref{eq:1aa} and \eqref{eq:1ac} we obtain
$$
i^2 < (d-\lambda^*)(a+d) - (d-\lambda^*)d = (d-\lambda^*)a
$$
and
$$
\frac{i^2+\lambda_ia}{ad} = \frac{i^2+\lambda^*a}{ad} < \frac{(d-\lambda^*)a+\lambda^*a}{ad} = 1.
$$
We may thus deduce that
$$
\left\lfloor\frac{i^2+\lambda_ia}{ad}\right\rfloor = 0.
$$
Finally, since
$$
(i+d)^2 = \left(\mu d+\alpha_{j+1}\right)^2 > (d-\lambda^*)(a+d) = \left(\left(\left\lfloor\frac{i^2+\lambda_i a}{ad}\right\rfloor+1\right)d-\lambda_i\right)(a+d),
$$
we deduce, from Proposition~\ref{prop2}, that $(a-i)^2\not\in \left\langle a,a+d\right\rangle$, as desired.
\end{proof}

Unfortunately, the value of $\SG{\left\langle a,a+d\right\rangle}$ given in the above corollary does not hold in general (if the condition $a+d\ge4d^3$ is not satisfied). However, as we will see below,  the number of values of $a$ not holding the equality $\SG{ \left\langle a,a+d\right\rangle} = \ub{a,d,1}$ is finite for each fixed $d$.

\section{Formulas for \texorpdfstring{$\ger{a,a+d}$}{<a,a+d>} with small \texorpdfstring{$d\ge3$}{d>=3}}\label{sec:2gen345}

In this section, we investigate the value of $\SG{\ger{a, a+d}}$ when $d$ is small.

For any positive integer $d\ge3$, we may define the set $E(d)$ to be the set of integers $a$ coprime to $d$ not holding the equality of Corollary~\ref{cor:ap}, that is, 
$$
E(d) := \left\{ a\in\N\setminus\{0,1\}\ \middle|\ \gcd(a,d)=1\text{ and } \SG{\ger{a,a+d}} \neq \ub{a,d,1}\right\}.
$$
Since $\lambda^*\le d-1$ then, from Corollary~\ref{cor:ap}, we obtain that $E(d)\subset[2,4d^3-1]\cap\N$.  We completely determine the set $E(d)$ for a few values of $d\ge3$ by computer calculations, see Table~\ref{tab1}.
 
\begin{table}[htbp] 
\begin{center}
\begin{tabular}{|c|c|c|}
\hline
$d$ & $|E(d)|$ & $E(d)$ \\
\hline\hline
$3$ & $0$ & $\emptyset$ \\ 
\hline
$4$ & $0$ & $\emptyset$ \\ 
\hline
$5$ & $5$ & $\left\{2,4,13,27,32\right\}$ \\
\hline
$6$ & $0$ & $\emptyset$ \\ 
\hline
$7$ & $10$ & $\left\{2,3,4,9,16,18,19,23,30,114\right\}$ \\
\hline
$8$ & $5$ & $\left\{5,9,21,45,77\right\}$ \\
\hline
$9$ & $5$ & $\left\{2,4,7,8,16\right\}$ \\
\hline
$10$ & $14$ & $\left\{3,9,13,23,27,33,43,123,133,143,153,163,333,343\right\}$ \\
\hline
$11$ & $14$ & $\left\{2,3,4,5,7,8,9,14,16,18,25,36,38,47\right\}$ \\
\hline
$12$ & $9$ & $\left\{13,19,25,31,67,79,139,151,235\right\}$ \\
\hline
\end{tabular}
\end{center}
\caption{$E(d)$ for the first values of $d\ge3$.}\label{tab1}
\end{table}
 
The exact values of $\SG{\ger{a,a+d}}$ when $a\in E(d)$, for $d\in\{3,\ldots,12\}$, are given in Appendix~\ref{app:2gen}.
\smallskip

For each value $d\in\left\{3,\ldots,12\right\}$, an explicit formula for $\SG{\ger{a, a+d}}$ can be presented excluding the values given in Table~\ref{tab1}. The latter can be done by using (essentially) the same arguments as those applied in the proofs of Theorem~\ref{theo:ap} and Corollary~\ref{cor:ap}. We present the proof for the case $d=3$.  

\begin{thm}\label{thm:2gen3}
Let $a\ge 2$ be an integer not divisible by $3$ and let $S=\ger{a, a+3}$. Then,
$$
\SG{S} = \left\{
\begin{array}{ll}
(a-(3b-1))^2 & \text{if either } (3b+1)^2 \le a+3 < (3b+2)^2 \hspace*{.4cm} \text{ and } a\equiv1\bmod{3}\\ 
&\hspace*{.9cm} \text{or } \ (3b+1)^2 \le 2(a+3) < (3b+2)^2 \text{ and } a\equiv2\bmod{3},\\ \ \\
(a-(3b+1))^2 & \text{if either } (3b+2)^2 \le a+3 < (3b+4)^2 \hspace*{.4cm} \text{ and } a\equiv1\bmod{3} \\ 
&\hspace*{.9cm} \text{or }\ (3b+2)^2 \le 2(a+3) < (3b+4)^2 \text{and } a\equiv2\bmod{3}.
\end{array}
\right.
$$
\end{thm}

\begin{proof}
Since  $\Frob{S}=(a-1)(a+2)-1=a^2+a-3<(a+1)^2$ then
$$
\SG{S}\le (a-1)^2.
$$
By Proposition~\ref{prop2}, we know that
\begin{equation}\label{eq:d3}
(a-i)^2 \in S \quad\Longleftrightarrow\quad (i+3)^2 \le \left(3\left\lfloor\frac{i^2+\lambda_i a}{3a}\right\rfloor+3-\lambda_i\right)(a+3),
\end{equation}
where $\lambda_i\in\{0,1,2\}$ such that $\lambda_i a +i^2 \equiv 0\bmod{3}$, that is,
$$
\lambda_i = \left\{\begin{array}{ll}
0 & \text{if } i\equiv 0\bmod{3} \text{ and } a\equiv1,2\bmod{3},\\
1 & \text{if } i\equiv 1,2\bmod{3} \text{ and } a\equiv2\bmod{3},\\
2 & \text{if } i\equiv 1,2\bmod{3} \text{ and } a\equiv1\bmod{3}.
\end{array}\right.
$$

We have four cases.
\setcounter{case}{0}

\begin{case}
Suppose that $a\equiv1\bmod{3}$ with $(3b+1)^2 \le a+3 < (3b+2)^2$. Note that $b\ge1$ since $a+3\ge 19$.\\
If $i\le 3b-2$ then
$$
(i+3)^2 \le (3b+1)^2 \le a+3 \le\left(3\left\lfloor\frac{i^2+\lambda_i a}{3a}\right\rfloor+3-\lambda_i\right)(a+3),
$$
obtaining, by equation~\eqref{eq:d3}, that $(a-i)^2\in S$. 
\smallskip

If $i=3b-1$ then
$$
i^2 = (3b-1)^2 = 9b^2-6b+1 \stackrel{b\ge1}{<} 9b^2+6b-2 = (3b+1)^2-3 \le a,
$$
obtaining that
$$
0\le \frac{i^2+\lambda_i a}{3a} = \frac{i^2+2a}{3a} < 1 \ \ (\text{since } 3b-1\equiv 2\bmod{3})
$$
and thus
$$
\left\lfloor\frac{i^2+\lambda_i a}{3a}\right\rfloor =0.
$$
Moreover, since
$$
\left(3\left\lfloor\frac{i^2+\lambda_i a}{3a}\right\rfloor+3-\lambda_i\right)(a+3) = a+3 < (3b+2)^2 = (i+3)^2,
$$
by equation \eqref{eq:d3}, we have that $(a-i)^2\notin S$.
\end{case}

\begin{case}
Suppose that $a\equiv1\bmod{3}$ with $(3b+2)^2 \le a+3 < (3b+4)^2$. If $b=0$, we have $(a-1)\notin S$ since
$$
\left(3\left\lfloor\frac{1+\lambda_i a}{3a}\right\rfloor+3-\lambda_i\right)(a+3) = a+3 < 4^2.
$$
Therefore $\SG{S}=(a-1)^2$ in this case. Suppose now that $b\ge1$.\\
If $i\le 3b-1$, then
$$
(i+3)^2 \le (3b+2)^2 \le a+3 \le\left(3\left\lfloor\frac{i^2+\lambda_i a}{3a}\right\rfloor+3-\lambda_i\right)(a+3),
$$
obtaining, by equation~\eqref{eq:d3}, that $(a-i)^2\in S$. 
\smallskip

If $i=3b$ then, using that $\lambda_i=0$,
$$
(i+3)^2 = (3b+3)^2 \le 3(3b+2)^2 \le 3(a+3) \le \left(3\left\lfloor\frac{i^2+\lambda_i a}{3a}\right\rfloor+3-\lambda_i\right)(a+3),
$$
obtaining, by equation~\eqref{eq:d3}, that $(a-i)^2\in S$.
\smallskip

If  $i=3b+1$ then
$$
i^2 = (3b+1)^2 = 9b^2+6b+1 \stackrel{b\ge1}{<} 9b^2+12b+1 = (3b+2)^2-3 \le a,
$$
obtaining that
$$
0\le \frac{i^2+\lambda_i a}{3a} = \frac{i^2+2a}{3a} < 1\ \ (\text{since } 3b+1\equiv 1\bmod{3})
$$
and thus
$$
\left\lfloor\frac{i^2+\lambda_i a}{3a}\right\rfloor =0.
$$
Moreover, since
$$
\left(3\left\lfloor\frac{i^2+\lambda_i a}{3a}\right\rfloor+3-\lambda_i\right)(a+3) = a+3 < (3b+4)^2 = (i+3)^2,
$$
therefore, by equation~\eqref{eq:d3}, we have that  $(a-i)^2\notin S$.
\end{case}

\begin{case}
Suppose that $a\equiv2\bmod{3}$ with $(3b+1)^2 \le 2(a+3) < (3b+2)^2$. Note that $b\ge1$ since $2(a+3)\ge 16$.\\
If $i\le 3b-2$ then
$$
(i+3)^2 \le (3b+1)^2 \le 2(a+3) \stackrel{\lambda_i\le 1}{\le}\left(3\left\lfloor\frac{i^2+\lambda_i a}{3a}\right\rfloor+3-\lambda_i\right)(a+3),
$$
obtaining, by equation~\eqref{eq:d3}, that $(a-i)^2\in S$. 
\smallskip

If $i=3b-1$ then
$$
i^2 = (3b-1)^2 = 9b^2-6b+1 \stackrel{b\ge1}{<} 9b^2+6b-5 = (3b+1)^2-6 \le 2a,
$$
obtaining that
$$
0\le \frac{i^2+\lambda_i a}{3a} = \frac{i^2+a}{3a} < 1\ \ (\text{since } 3b-1\equiv 2\bmod{3})
$$
and thus
$$
\left\lfloor\frac{i^2+\lambda_i a}{3a}\right\rfloor =0.
$$
Moreover, since
$$
\left(3\left\lfloor\frac{i^2+\lambda_i a}{3a}\right\rfloor+3-\lambda_i\right)(a+3) = 2(a+3) < (3b+2)^2 = (i+3)^2,
$$
therefore, by equation~\eqref{eq:d3}, we have that $(a-i)^2\notin S$.
\end{case}

\begin{case}
Suppose that $a\equiv2\bmod{3}$ with $(3b+2)^2 \le 2(a+3) < (3b+4)^2$. If $b=0$, we have $(a-1)\notin S$ since
$$
\left(3\left\lfloor\frac{1+\lambda_i a}{3a}\right\rfloor+3-\lambda_i\right)(a+3) = 2(a+3) < 4^2.
$$
Therefore $\SG{S}=(a-1)^2$ in this case. Suppose now that $b\ge1$.\\
If $i\le 3b-1$ then
$$
(i+3)^2 \le (3b+2)^2 \le 2(a+3) \le\left(3\left\lfloor\frac{i^2+\lambda_i a}{3a}\right\rfloor+3-\lambda_i\right)(a+3),
$$
obtaining, by equation~\eqref{eq:d3}, that $(a-i)^2\in S$. 
\smallskip

If $i=3b$ then, using that $\lambda_i=0$,
$$
(i+3)^2 = (3b+3)^2 \stackrel{b\ge1}{<} \frac{3}{2}(3b+2)^2 \le 3(a+3) \le \left(3\left\lfloor\frac{i^2+\lambda_i a}{3a}\right\rfloor+3-\lambda_i\right)(a+3).
$$
Therefore, by equation~\eqref{eq:d3}, we have $(a-i)^2\in S$. 
\smallskip

If $i=3b+1$ then
$$
i^2 = (3b+1)^2 = 9b^2+6b+1 \stackrel{b\ge1}{<} 9b^2+12b-2 = (3b+2)^2-6 \le 2a,
$$
when $b\ge1$ and clearly $i^2=1<2a$ when $b=0$, obtaining that
$$
0\le \frac{i^2+\lambda_i a}{3a} = \frac{i^2+a}{3a} < 1
$$
and
$$
\left\lfloor\frac{i^2+\lambda_i a}{3a}\right\rfloor =0.
$$
Moreover, since
$$
\left(3\left\lfloor\frac{i^2+\lambda_i a}{3a}\right\rfloor+3-\lambda_i\right)(a+3) = 2(a+3) < (3b+4)^2 = (i+3)^2,
$$
therefore, by equation\eqref{eq:d3}, we have that $(a-i)^2\notin S$.
\end{case}
\end{proof}

The proofs of the following two theorems are completely analogous to that of Theorem~\ref{thm:2gen3} with a larger number of cases to be analyzed (in each case, the appropriate inequality is obtained in order to apply Proposition~\ref{prop2}).  

\begin{thm}\label{thm:2gen4}
Let $a\ge 3$ be an odd integer and let $S=\ger{a, a+4}$. Then,
$$
\SG{S} = \left\{
\begin{array}{ll}
(a-(4b-1))^2 & \text{if either } (4b+1)^2 \le a+4 < (4b+3)^2 \hspace*{.4cm} \text{ and } a\equiv1\bmod{4} \\
&\hspace*{.9cm} \text{or } \ (4b+1)^2 \le 3(a+4) < (4b+3)^2 \text{and } a\equiv3\bmod{4},\\ \ \\
(a-(4b+1))^2 & \text{if either } (4b+3)^2 \le a+4 < (4b+5)^2 \hspace*{.4cm} \text{ and } a\equiv1\bmod{4}\\
&\hspace*{.9cm} \text{or } \ (4b+3)^2 \le 3(a+4) < (4b+5)^2 \text{and } a\equiv3\bmod{4}.\\ \ \\
\end{array}
\right.
$$
\end{thm}

\begin{thm}\label{thm:2gen5}
Let $a\ge 2$ be an integer not divisible by $5$ and let $S=\ger{a, a+5}$. Then,
{\footnotesize $$
\SG{S} = \left\{
\begin{array}{ll}
1 & \text{if } a=2 \text{ or } 4,\\
10^2 & \text{if } a=13,\\
(a-6)^2 & \text{if } a=27 \text{ or } 32,\\
(a-(5b-2))^2 &\text{if either } (5b+2)^2 \le a+5 < (5b+3)^2 \hspace*{.4cm} \text{ and } a\equiv4\bmod{5}\\
&\hspace*{.9cm} \text{or }\ (5b+2)^2 \le 2(a+5) < (5b+3)^2 \text{and } a\equiv2\bmod{5},\\ \ \\
(a-(5b-1))^2 &\text{if either } (5b+1)^2 \le a+5 < (5b+4)^2 \hspace*{.4cm} \text{ and } a\equiv1\bmod{5} \\
&\hspace*{.9cm} \text{or } \ (5b+1)^2 \le 2(a+5) < (5b+4)^2 \text{and } a\equiv3\bmod{5}, a\neq 13,\\ \ \\
(a-(5b+1))^2 &\text{if either } (5b+4)^2 \le a+5 < (5b+6)^2 \hspace*{.4cm} \text{ and } a\equiv1\bmod{5} \\
&\hspace*{.9cm} \text{or }\ (5b+4)^2 \le 2(a+5) < (5b+6)^2 \text{and } a\equiv3\bmod{5},\\ \ \\
(a-(5b+2))^2 &\text{if either } (5b+3)^2 \le a+5 < (5b+7)^2 \hspace*{.4cm} \text{ and } a\equiv4\bmod{5},  a\neq 4\\\
&\hspace*{.9cm} \text{or }\ (5b+3)^2 \le 2(a+5) < (5b+7)^2 \text{and } a\equiv2\bmod{5}, a\neq 2,27,32.
\end{array}
\right.
$$}
\end{thm}

\section{Study of \texorpdfstring{$\ger{a,a+1}$}{<a,a+1>}}\label{sec:2gen1}

We investigate the square Frobenius number of $\ger{a,a+1}$ with $a\ge 2$. We first study the case when neither $a$ nor $a+1$ is a square integer.

\begin{prop}\label{prop:2gen1}
Let $a$ be a positive integer such that $b^2<a<a+1<(b+1)^2$ for some integer $b\ge 1$. Then,
$$
\SG{\ger{a,a+1}} = (a-b)^2.
$$
\end{prop}

\begin{proof}
Since $\Frob{\ger{a,a+1}}=a^2-a-1$ then
$$
(a-1)^2 \le \Frob{\ger{a,a+1}} < a^2.
$$

We thus have that $\SG{\ger{a,a+1}} < a^2$. We shall show that $(a-i)^2\in \ger{a,a+1}$ for $i\in\left\{1,2,\ldots,b-1\right\}$.

We first observe that 
\begin{equation}\label{eq1}
(a-i)^2 = a^2-2ai+i^2 = (a-2i)a + i^2 = (a-2i-i^2)a + i^2(a+1),
\end{equation}
for any integer $i$.

Since for any $i\in\left\{1,2,\ldots,b-1\right\}$ we have
$$
a-2i-i^2 = a-i(i+2) \ge a-(b-1)(b+1) = a-b^2+1 > 0
$$
and
$$
i^2>0
$$  
then, by \eqref{eq1}, we deduce that $(a-i)^2\in \left\langle a,a+1 \right\rangle$ for any $i\in\left\{1,2,\ldots,b-1\right\}$. 
\smallskip

Finally, since $a+1<(b+1)^2$ (implying that $a-2b-b^2<0$) and $0< b^2<a$ then we may deduce, from \eqref{eq1}, that $(a-b)^2\not\in\ger{a,a+1}$. 
\end{proof}

Let $(u_n)_{n\ge1}$ be the recursive sequence defined by 
\begin{equation}\label{eq:recursive}
u_1=1, u_2=2, u_3=3, u_{2n} = u_{2n-1}+u_{2n-2}\text{ and } u_{2n+1} = u_{2n}+u_{2n-2}\text{ for all } n\ge2.
\end{equation} 
\smallskip

The first few values of $(u_n)_{n\ge1}$ are
$$
1 , 2 , 3 , 5 , 7 , 12 , 17 , 29 , 41 , 70 , 99 , 169 , 239 , 408 , 577 , 985 , \ldots\ldots
$$

This sequence appears in a number of other contexts. For instance, it corresponds to the {\em denominators of Farey fraction approximations to $\sqrt{2}$}, where the fractions are $\frac{1}{1}$, $\frac{2}{1}$, $\frac{3}{2}$, $\frac{4}{3}$, $\frac{7}{5}$, $\frac{10}{7}$, $\frac{17}{12}$, $\frac{24}{17}\dots$, see \cite{oeis}.
\smallskip

We pose the following conjecture in the case when either $a$ or $a+1$ is a square integer.

\begin{conj}\label{conj:2gen1}
Let $(u_n)_{n\ge1}$ be the recursive sequence given in \eqref{eq:recursive}.
\smallskip

\noindent 
If $a=b^2$ for some integer $b\ge 1$ then
$$
\SG{\ger{a,a+1}} = \left\{
\begin{array}{ll}
\left(a-\left\lfloor b\sqrt{2} \right\rfloor\right)^2 & \text{if } b\not\in\displaystyle\bigcup_{n\ge0}\left\{u_{4n+1},u_{4n+2}\right\}, \\ \ \\
\left(a-\left\lfloor b\sqrt{3} \right\rfloor\right)^2 & \text{if } b\in\displaystyle\bigcup_{n\ge0}\left\{u_{4n+1},u_{4n+2}\right\}.
\end{array}
\right.
$$
If $a+1=b^2$ for some integer $b\ge 1$ then
$$
\SG{\ger{a,a+1}} = \left\{
\begin{array}{ll}
\left(a-\left\lfloor b\sqrt{2} \right\rfloor\right)^2 & \text{if } b\not\in\displaystyle\bigcup_{n\ge1}\left\{u_{4n-1},u_{4n}\right\}, \\ \ \\
\left(a-\left\lfloor b\sqrt{3} \right\rfloor\right)^2 & \text{if } b\in\displaystyle\bigcup_{n\ge1}\left\{u_{4n},u_{4n+3}\right\}, \\ \ \\
2^2 & \text{if } b = u_3 = 3.
\end{array}
\right.
$$
\end{conj}

The formulas of Conjecture~\ref{conj:2gen1} have been verified by computer for all integers $a\ge 2$ up to $10^6$.

\section{Study of \texorpdfstring{$\ger{a,a+2}$}{<a,a+2>}}\label{sec:2gen2}

We investigate the square Frobenius number of $\ger{a,a+2}$ with $a\ge3$ odd. We first study the case when neither $a$ nor $a+2$ is a square integer.

\begin{prop}\label{prop:2gen2}
Let $a\ge3$ be an odd integer such that $(2b+1)^2<a<a+2<(2b+3)^2$ for some integer $b\ge 1$. Then,
$$
\SG{\ger{a,a+2}} = (a-(2b+1))^2.
$$
\end{prop}

\begin{proof}
Since $\Frob{\ger{a,a+2}}=(a-1)(a+1)-1=a^2-2$ then
$$
(a-1)^2 < \Frob{\ger{a,a+2}} < a^2.
$$
We thus have that $\SG{\ger{a,a+2}} < a^2$. We shall show that $(a-i)^2\in\ger{a,a+2}$ for $i\in\left\{1,2,\ldots,2b\right\}$.

We first observe that for any integer $i$, we have
\begin{equation}\label{eq2}
(a-2i)^2 = a^2-4ai+4i^2 = (a-4i)a + 4i^2 = (a-4i-2i^2)a + 2i^2(a+2).
\end{equation}

Since for any $i\in\left\{1,2,\ldots,b\right\}$ we have
$$
a-4i-2i^2 = a-2i(i+2) \ge a-2i(2i+1) > a-(2i+1)^2 \ge a-(2b+1)^2 > 0
$$
and
$$
2i^2>0
$$
then, by \eqref{eq2}, it follows that $(a-2i)^2\in\ger{a,a+2}$ for any $i\in\left\{1,2,\ldots,b\right\}$.
\smallskip

Moreover, for any integer $i$, we have
\begin{equation}\label{eq3}
(a-(2i+1))^2 \begin{array}[t]{l}
= a^2 -2a(2i+1) + (2i+1)^2 = (a-2(2i+1))a + (2i+1)^2 \\
= (a-4i-3)a + (2i+1)^2+a \\
= \left(a-4i-3-\frac{(2i+1)^2+a}{2}\right)a + \frac{(2i+1)^2+a}{2}(a+2) \\
= \frac{a-4i^2-12i-7}{2}a + \frac{(2i+1)^2+a}{2}(a+2) \\
= \frac{a+2-(2i+3)^2}{2}a + \frac{(2i+1)^2+a}{2}(a+2).
\end{array}
\end{equation}
Note that $a+2-(2i+3)^2$ and $(2i+1)^2+a$ are even because $a$ is odd.

Since, for any $i\in\left\{0,1,\ldots,b-1\right\}$ we have
$$
\frac{a+2-(2i+3)^2}{2} \ge \frac{a+2-(2b+1)^2}{2} >0
$$
and
$$
\frac{(2i+1)^2+a}{2}>0
$$
then it follows, from \eqref{eq3} ,that $(a-(2i+1))^2\in\ger{a,a+2}$, for any $i\in\left\{0,1,\ldots,b-1\right\}$. 
\smallskip

Finally, since
$$
0<\frac{(2b+1)^2+a}{2}<a
$$
and
$$
\frac{a+2-(2b+3)^2}{2} <0,
$$
then we have, from \eqref{eq3}, that $(a-(2b+1))^2\not\in\ger{a,a+2}$.
\end{proof}

We pose the following conjecture in the case when either $a$ or $a+2$ is a square integer.
 
\begin{conj}\label{conj:2gen2}
Let $(u_n)_{n\ge1}$ be the recursive sequence given in \eqref{eq:recursive}.
\smallskip

\noindent
If  $a=(2b+1)^2$ for some integer $b\ge 1$ then 
$$
\SG{\ger{a,a+2}} = \left\{
\begin{array}{cl}
\left(a-2\left\lfloor\frac{(2b+1)\sqrt{2}}{2}\right\rfloor\right)^2 & \text{if } (2b+1)\not\in\displaystyle\bigcup_{n\ge1}\left\{u_{4n+1}\right\}, \\ \ \\
\left(a-\left\lfloor (2b+1)\sqrt{3} \right\rfloor\right)^2 & \text{if } (2b+1)\in\displaystyle\bigcup_{n\ge2}\left\{u_{4n+1}\right\}, \\ \ \\
38^2 & \text{if } 2b+1=u_5=7.
\end{array}
\right.
$$
If  $a+2=(2b+1)^2$ for some integer $b\ge 1$ then 
$$
\SG{\ger{a,a+2}} = \left\{
\begin{array}{ll}
\left(a-2\left\lfloor\frac{(2b+1)\sqrt{2}}{2}\right\rfloor\right)^2 & \text{if } (2b+1)\not\in\displaystyle\bigcup_{n\ge0}\left\{u_{4n+3}\right\}, \\ \ \\
\left(a-\left\lfloor (2b+1)\sqrt{3} \right\rfloor\right)^2 & \text{if } (2b+1)\in\displaystyle\bigcup_{n\ge0}\left\{u_{4n+3}\right\}.
\end{array}
\right.
$$
\end{conj}

The formulas of Conjecture~\ref{conj:2gen2} have been verified by computer for all odd integers $a\ge 3$ up to $10^6$.

\section{Concluding remarks}\label{sec:concluding}

In the process of investigating square Frobenius numbers different problems arose.  
We naturally consider the $P$-type function $\KRR{S}=\KR{S}$ defined as,
$$
\KR{S}:= \text{ the smallest perfect $k$-power integer belonging to } S.
$$
It is clear that 
\begin{equation}\label{eq;MG}
s\le \KR{S}\le s^k
\end{equation}
where $s$ is the multiplicity of $S$.

\begin{thm}
Let $S_A=\ger{a,a+d,\ldots,a+kd}$ where $a,d,k$ are positive integers with $\gcd (a,d)=1$. 
If $d\le \frac{ak}{1+2k}$ then 
$$
\SR{S_A}\le (a-d)^2.
$$
\end{thm}

\begin{proof}
We shall use the characterization given in Proposition \ref{prop2} with $i=d$. In this case $\lambda_d=0$ and $d\le\frac{ak}{1+2k}<a$ thus
$$
\left(\left(\left\lfloor\frac{d^2+0 a}{ad}\right\rfloor+k\right)d-0\right)(a+kd)=\left(\left(\left\lfloor\frac{d}{a}\right\rfloor+k\right)d-0\right)(a+kd)= akd+(kd)^2.
$$
Thus, 
$$
(d+kd)^2\le akd+(kd)^2\iff d^2+2kd^2\le akd \iff d\le \frac{ak}{1+2k}.
$$ 
Therefore, by Proposition \ref{prop2}, $(a-d)^2\in S_A$.
\end{proof}

\begin{prob}
Let $k\ge 2$ be an integer and let $S$ be a numerical semigroup. Investigate the computational complexity to determine $\GG{S}$ and/or $\KR{S}$. 
\end{prob}

Or more ambitious,

\begin{quest}
Let $k\ge 2$ be an integer. Is there a closed formula for $\GG{S}$ and/or $\KR{S}$ for any semigroup $S$?
\end{quest}

Perhaps a first step on this direction might be the following.

\begin{prob}
Give a formula for $\SG{\ger{F_i, F_j}}$ and/or $\SR{\ger{F_i, F_j}}$ with $\gcd(F_i,F_j)=1$ where $F_k$ denotes the $k^{th}$ Fibonacci number. What about $\SG{\ger {a^2,b^2}}$ where $a$ and $b$ are relatively prime integers ? We clearly have that $\SR{\ger {a^2,b^2}}=a^2$ for $1\le a<b$.
\end{prob}


\newpage
\appendix
\section{Complement to formulas for \texorpdfstring{$\ger{a,a+d}$}{<a,a+d>} with small \texorpdfstring{$d\ge3$}{d>=3}}\label{app:2gen}
In Tabular~\ref{tab2}, we compare the exact values of $\SG{\ger{a,a+d}}$ and the formula $\ub{a,d,1}$, when $a\in E(d)$ for $d\in\{3,\ldots,12\}$.

\begin{table}[htbp]
\begin{center}
\begin{tabular}{cc}
\begin{tabular}{|c|c|c|c|}
\hline
$d$ & $a$ & $\SG{\ger{a,a+d}}$ & $\ub{a,d,1}$ \\
\hline
\hline
$5$ & $2$ & $1$ & $0$ \\
\hline
$5$ & $4$ & $1$ & $2^2$ \\
\hline
$5$ & $13$ & $10^2$ & $9^2$ \\
\hline
$5$ & $27$ & $21^2$ & $20^2$ \\
\hline
$5$ & $32$ & $26^2$ & $25^2$ \\
\hline
$7$ & $2$ & $1$ & $5^2$ \\
\hline
$7$ & $3$ & $2^2$ & $0$ \\
\hline
$7$ & $4$ & $5^2$ & $6^2$ \\
\hline
$7$ & $9$ & $7^2$ & $6^2$ \\
\hline
$7$ & $16$ & $14^2$ & $13^2$ \\
\hline
$7$ & $18$ & $17^2$ & $16^2$ \\
\hline
$7$ & $19$ & $14^2$ & $13^2$ \\
\hline
$7$ & $23$ & $21^2$ & $20^2$ \\
\hline
$7$ & $30$ & $28^2$ & $27^2$ \\
\hline
$7$ & $114$ & $105^2$ & $104^2$ \\
\hline
$8$ & $5$ & $4^2$ & $3^2$ \\
\hline
$8$ & $9$ & $10^2$ & $12^2$ \\
\hline
$8$ & $21$ & $16^2$ & $15^2$ \\
\hline
$8$ & $45$ & $36^2$ & $35^2$ \\
\hline
$8$ & $77$ & $64^2$ & $63^2$ \\
\hline
$9$ & $2$ & $3^2$ & $4^2$ \\
\hline
$9$ & $4$ & $3^2$ & $6^2$ \\
\hline
$9$ & $7$ & $6^2$ & $11^2$ \\
\hline
$9$ & $8$ & $6^2$ & $4^2$ \\
\hline
$9$ & $16$ & $9^2$ & $12^2$ \\
\hline
$10$ & $3$ & $2^2$ & $7^2$ \\
\hline
$10$ & $9$ & $7^2$ & $12^2$ \\
\hline
$10$ & $13$ & $10^2$ & $9^2$ \\
\hline
$10$ & $23$ & $20^2$ & $19^2$ \\
\hline
$10$ & $27$ & $26^2$ & $25^2$ \\
\hline
$10$ & $33$ & $30^2$ & $29^2$ \\
\hline
\end{tabular}
&
\begin{tabular}{|c|c|c|c|}
\hline
$d$ & $a$ & $\SG{\ger{a,a+d}}$ & $\ub{a,d,1}$ \\
\hline
\hline
$10$ & $43$ & $40^2$ & $39^2$ \\
\hline
$10$ & $123$ & $110^2$ & $109^2$ \\
\hline
$10$ & $133$ & $120^2$ & $119^2$ \\
\hline
$10$ & $143$ & $130^2$ & $129^2$ \\
\hline
$10$ & $153$ & $140^2$ & $139^2$ \\
\hline
$10$ & $163$ & $150^2$ & $149^2$ \\
\hline
$10$ & $333$ & $310^2$ & $309^2$ \\
\hline
$10$ & $343$ & $320^2$ & $319^2$ \\
\hline
$11$ & $2$ & $3^2$ & $4^2$ \\
\hline
$11$ & $3$ & $5^2$ & $9^2$ \\
\hline
$11$ & $4$ & $5^2$ & $6^2$ \\
\hline
$11$ & $5$ & $7^2$ & $9^2$ \\
\hline
$11$ & $7$ & $4^2$ & $2^2$ \\
\hline
$11$ & $8$ & $7^2$ & $12^2$ \\
\hline
$11$ & $9$ & $8^2$ & $12^2$ \\
\hline
$11$ & $14$ & $13^2$ & $19^2$ \\
\hline
$11$ & $16$ & $14^2$ & $20^2$ \\
\hline
$11$ & $18$ & $15^2$ & $13^2$ \\
\hline
$11$ & $25$ & $22^2$ & $20^2$ \\
\hline
$11$ & $36$ & $33^2$ & $31^2$ \\
\hline
$11$ & $38$ & $36^2$ & $34^2$ \\
\hline
$11$ & $47$ & $44^2$ & $42^2$ \\
\hline
$12$ & $13$ & $14^2$ & $18^2$ \\
\hline
$12$ & $19$ & $17^2$ & $16^2$ \\
\hline
$12$ & $25$ & $26^2$ & $30^2$ \\
\hline
$12$ & $31$ & $29^2$ & $28^2$ \\
\hline
$12$ & $67$ & $59^2$ & $58^2$ \\
\hline
$12$ & $79$ & $71^2$ & $70^2$ \\
\hline
$12$ & $139$ & $125^2$ & $124^2$ \\
\hline
$12$ & $151$ & $137^2$ & $136^2$ \\
\hline
$12$ & $235$ & $215^2$ & $214^2$ \\
\hline
\end{tabular}\\ \ \\
\end{tabular}
\end{center}
\caption{$\SG{\ger{a,a+d}}$ and $\ub{a,d,1}$ when $a\in E(d)$ for $d\in\{3,\ldots,12\}$}\label{tab2}
\end{table}


\begin{thebibliography}{99} 

\bibitem{ramirez1}
J.L. Ram\'irez Alfons\'in, Complexity of the Frobenius problem, {\em Combinatorica} {\bf 16}(1) (1996), 143-147.

\bibitem{ramirez2}
J.L. Ram\'irez Alfons\'in, The Diophantine Frobenius Problem, Oxford Lecture Ser. in Math. and its Appl. 30, Oxford University Press 2005.

\bibitem{Rob} J.B. Roberts, Note on linear forms, {\em Proc. Amer. Math. Soc.} {\bf 7} (1956), 465-469.

\bibitem{Rod} \O.J. R\o dseth, On a linear diophantine problem of Frobenius II, {\em J. Reine Angew. Math.} {\bf 307/308} (1979), 431-440.

\bibitem{oeis} The On-line Encyclopedia of Integers Sequences, \url{https://oeis.org/A002965}

\end{thebibliography}
\end{document}